\documentclass[11pt, a4paper]{amsart}
\usepackage{amsfonts,amssymb,amsmath,amsthm,amscd,mathtools,multicol,tikz, tikz-cd,caption,enumerate,mathrsfs,thmtools,cite}
\usepackage{inputenc}
\usepackage[foot]{amsaddr}
\usepackage[pagebackref=true, colorlinks, linkcolor=blue, citecolor=red]{hyperref}
\usepackage{latexsym}
\usepackage{fullpage}
\usepackage{microtype}
\usepackage{subfiles} 

\usepackage{palatino}
\makeatletter
\makeindex 
\newcommand{\be}{\begin{equation}}
\newcommand{\ee}{\end{equation}}
\newcommand{\beano}{\begin{eqn*}} 
	\newcommand{\eeano}{\end{eqnarray*}}
\newcommand{\ba}{\begin{array}}
	\newcommand{\ea}{\end{array}}

\declaretheoremstyle[headfont=\normalfont]{normalhead}

\newtheorem{theorem}{Theorem}[section]

\newtheorem{lemma}[theorem]{Lemma}

\newtheorem{remark}[theorem]{Remark}
\newtheorem{example}[theorem]{Example}

\numberwithin{equation}{section}

\begin{document}
\title{Decomposition of quandle rings of dihedral quandles}
\author{Dilpreet Kaur}
\email{dilpreetkaur@iitj.ac.in}
\address{Indian Institute of Technology Jodhpur}

\author{Pushpendra Singh}
\email{singh.105@iitj.ac.in}
\address{Indian Institute of Technology Jodhpur}

\thanks{We are thankful to Amit Kulshrestha for his useful comments on our work.}
\maketitle

\begin{abstract}
Let $K = \mathbb{R}$ or $\mathbb{C}$ and $\mathcal{R}_n$ be the dihedral quandle of order $n.$ In this article, we give decomposition of the quandle ring $K[\mathcal{R}_n]$ into indecomposable right $K[\mathcal{R}_n]$-modules for all even $n \in \mathbb N$. It follows that the decomposition of $K[\mathcal{R}_n]$ given in \cite[Prop. 4.18(2)]{EFT} is valid only in the case when $n$ is not divisible by $4$.
\end{abstract}

\section{Introduction}\label{1}
A set $X$ along with a binary operation $\triangleright$ is called a {\it quandle} if it satisfies the following three axioms:
\begin{enumerate}[(1).]
 \item $x \triangleright x = x$ for all $x \in X.$
 \item For any pair $x,y\in X,$ there exist a unique $z\in X$ such that $x=z \triangleright y.$
 \item $(x \triangleright y) \triangleright z = (x \triangleright z)\triangleright (y \triangleright z)$ for all $x,y,z \in X.$
\end{enumerate}
The above three axioms are motivated by the three Reidemeister moves of diagrams of knots. 

In \cite{BPS19}, authors introduced the notion of quandle rings, which is analogous to the notion of group rings. Let $K$ be an associative ring and $(X, \triangleright)$ be a quandle. The quandle ring $K[X]$ as a set is same as  $\{ \sum_{i} \alpha_ix_i ~|~ \alpha_i\in K, x_i\in X\}.$ It forms an abelian group under pointwise addition. The following multiplication equips it with the structure of a ring.

$$\left(\sum_i \alpha_ix_i\right).\left(\sum_j \beta_jx_j\right) = \sum_{i,j} \alpha_i\beta_j (x_i \triangleright x_j).$$

Let $n\in \mathbb{N}$ and $\mathcal{R}_n=\{0,1,2,\dots, n-1\}.$ We define $i\triangleright j = 2j-i({\rm mod}(n)).$ The set $\mathcal{R}_n$ along with the binary operation $\triangleright$ forms a quandle, and it is known as {\it dihedral quandle.} 

In this article, we give the decomposition of $K[\mathcal{R}_n]$ into indecomposable right $K[\mathcal{R}_n]$-modules, where $\mathcal{R}_n$ denotes the dihedral quandle of even order $n.$ In \cite[Prop. 4.18]{EFT}, authors give decomposition of $K[\mathcal{R}_n]$ into indecomposable right $K[\mathcal{R}_n]$-modules. However, there is a defect in their decomposition of $K[\mathcal{R}_n]$ in the case when $n$ is divisible by $4$.

We conclude this article with a detailed discussion on decomposition of quandle ring $K[\mathcal{R}_8]$ of dihedral quandle of order $8$.
  
\section{Decomposition of quandle rings of dihedral quandles of even order}
Let $(X, \triangleright)$ be a quandle.
From the definition of quandle, we get that right multiplication maps $R_x : X \to X$ defined by $R_x(y)= y \triangleright x$;  for $x\in X$, are automorphisms of $X$. The group generated by $\{ R_x ~|~ x\in X\}$ is called group of inner automorphisms of $X$ and it is denoted by ${\rm Inn}(X).$ The group ${\rm Inn}(X)$ acts on the set $X$ via natural action. The quandle $X$ is called {\it connected} if the action ${\rm Inn}(X)$ on $X$ is transitive.   

\begin{example}\label{example}
\normalfont
\begin{enumerate}[(1).]
 \item The dihedral quandle $\mathcal{R}_n$ defined in \S\ref{1} is connected when $n$ is an odd positive integer (see \cite[Example 4.17]{EFT}).
 \item Let $n=2m$ and $\mathcal{R}_n$ be the dihedral quandle of order $n.$  The action of ${\rm Inn}(\mathcal{R}_n)$ on $\mathcal{R}_n$ is not transitive in this case. The orbits of action of ${\rm Inn}(\mathcal{R}_n)$ on $\mathcal{R}_n$ are $A_n:= \{0,2,\dots, n-2\}$ and $B_n:= \{1, 3, \dots, n-1\}.$  
\end{enumerate}
\end{example}

In \cite{EFT}, authors discussed the decomposition of $K[X]$ into indecomposable right $K[X]$-modules. We define $\rho : {\rm Inn}(X) \to {\rm GL}(K[X])$ by $\rho(R_x)(y) = y\triangleright x$ for all $R_x\in {\rm Inn}(X)$ and $y\in X.$ The map $\rho$ is a representation of ${\rm Inn}(X)$ and decomposing $K[X]$ into indecomposable right $K[X]$-modules is equivalent to decomposing $\rho$ into irreducible representations of the group ${\rm Inn}(X).$

Let $X$ be a quandle of order $n$ and $\mathcal S_n$ be the symmetric group of degree $n$. Suppose $X$ is not a connected quandle and $X=X_1\cup X_2 \cup \dots X_k,$ where $X_i$ is an orbit for action of ${\rm Inn}(X)$ on $X$. We recall from \cite{EFT} that the map $\Phi : X\to \mathcal S_n$ defined by $\Phi(x)=R_x$ restricts to the maps $\Phi_i : X \to \mathcal S_{n_i},$ where $n_i=|X_i|$ for all $1\leq i\leq k.$ In this case, $K[X]=\oplus_{i=1}^{k} K[X_i]$ and decomposition of $K[X_i]$ into indecomposable right $K[X_i]$-modules is equivalent to decomposing corresponding irreducible representations of the group ${\rm Inn}(X_i).$

From now onwards, $n=2m$ for $m\in \mathbb{N}.$
\begin{lemma}\label{dihedral}
 With the notations same as defined in Example \ref{example}(2), the groups ${\rm Inn}(A_n)$ and ${\rm Inn}(B_n)$ are isomorphic to the dihedral group of order $n.$ 
\end{lemma}

\begin{proof}
 We consider the right multiplication maps $R_j(i)=i\triangleright j =(2j-i)({\rm mod} n)$ for all $i, j \in \mathcal{R}_n.$ We observe that 
 $$\displaystyle R_j = \prod_{i=j+1}^{j+m-1}(i, 2j-i),$$
where $(i, 2j-i) \in \mathcal S_n$ is a transposition. Thus, $R_j=R_{j+m}=(R_1R_0)^jR_0$ for all $0\leq j\leq m-1$, and the group ${\rm Inn}(\mathcal{R}_n) $ of inner automorphisms of $\mathcal{R}_n$ is generated by $R_0$ and $R_1.$ 
 
 Let $S_0$ and $S_1$ be restrictions of $R_0$ and $R_1$ respectively, on the set $A_n$. It follows from the above discussion that the group ${\rm Inn}(A_n)$ is generated by $S_0$ and $S_1.$ Note that 
$S_0= (2,2m-2)(4,2m-4)\dots (m-2, m+2)$ and $S_1S_0=(0, 2, 4, \dots 2m-2).$ Therefore, ${\rm Inn}(A_n)$ is isomorphic to the dihedral group of order $n.$ A similar calculation holds for the group ${\rm Inn}(B_n)$.
\end{proof}

\remark\label{remark}
\normalfont 
We consider the following finite presentations for the dihedral group $D_m$ of order $2m$
$$D_m:= \langle r,s ~|~ r^m =s^2=1, srs^{-1}=r^{-1} \rangle.$$
Let $S_j$ and $T_j$ denote the restrictions of right multiplication map $R_j$ of quandle $\mathcal{R}_n$ on orbits $A_n$ and $B_n,$ respectively. The identification of $S_0$ with $s$ and 
of $S_1S_0$ with $r$ is an isomorphism between $D_m$ and ${\rm Inn}(A_n)$. Along the same lines, the identification of $T_0$ with $s$ and 
of $T_1T_0$ with $r$ is an isomorphism between $D_m$ and ${\rm Inn}(B_n)$. \hfill $\square$

Let $K=\mathbb{R}$ or $\mathbb{C}.$
By Example \ref{example}(2), $K[\mathcal{R}_n]=K[A_n]\oplus K[B_n].$
We know that $A_n$ is a basis of $K[A_n]$ and $B_n$ is that of
$K[B_n]$. We identify $A_n$ with $\{v_0, v_2, \dots , v_{n-2}\} \subseteq K[A_n]$ and $B_n$ with $\{v_1, v_3, \dots , v_{n-1}\} \subseteq K[B_n]$. Clearly, the modules generated by  $v_{triv, even}= \sum_{i=0}^{m-1} v_{2i}$ and $v_{triv, odd}= \sum_{i=0}^{m-1} v_{2i+1}$ are indecomposable modules over $K[A_n]$ and $K[B_n]$, respectively.
Moreover if $m$ is an even number, then the following lemma yields one dimensional indecomposable modules over $K[A_n]$ and $K[B_n].$
\begin{lemma}\label{ref}
 Let $m=2t, t\in \mathbb{N}.$ The modules generated by  $v_{ref, even}= \sum_{i=0}^{t-1} v_{4i}- \sum_{i=0}^{t-1}v_{4i+2}$ and $v_{ref, odd}= \sum_{i=0}^{t-1} v_{4i+1}-\sum_{i=0}^{t-1}v_{4i+3}$ are indecomposable modules over $K[A_n]$ and $K[B_n]$, respectively.
\end{lemma}

\begin{proof}
We consider the right multiplication map $R_j(i)=i\triangleright j =(2j-i)({\rm mod} n)$ for all $i, j \in \mathcal{R}_n.$ Let $S_j$ be the restriction of $R_j$ on the set $A_n.$ It is easy to see that $S_j(v_{ref, even})=v_{ref, even}$ if $j$ is even and $S_j(v_{ref, even})=-v_{ref, even}$ if $j$ is odd.

Let $T_j$ be the restriction of $R_j$ on the set $B_n.$ It is easy to see that $T_j(v_{ref, odd})=-v_{ref, odd}$ if $j$ is even and $T_j(v_{ref, odd})=v_{ref, odd}$ if $j$ is odd. This completes the proof of lemma. 
\end{proof}

Let $S_j$ and $T_j$ denote the restrictions of right multiplication map $R_j$ of quandle $\mathcal{R}_n$ on orbits $A_n$ and $B_n,$ respectively. 
Consider the representation $\rho_{A_n}: {\rm Inn}(A_n)\to {\rm GL}(K[A_n])$ defined by $\rho(S_j)(v_i)=v_{i\triangleright j}$; for all $S_j\in {\rm Inn}(A_n).$ Let $\chi_{A_n}$ be the character associated to $\rho_{A_n}$. Similarly we define $\rho_{B_n}$ and $\chi_{B_n}.$
With these notations, we have the following lemma.
\begin{lemma}\label{Character}
 \begin{enumerate}[(1).]
  \item If $m$ is an odd number, then
  $$\chi_{A_n}(g)=\chi_{B_n}(g)= \begin{cases}
                  m \quad \quad \text{if } g=1\\
                  1 \quad \quad \text { if } g=S_{i} \text{ or } T_{i}, ~ 0\leq i\leq n-1\\
                  0 \text \quad \quad {otherwise}
                 \end{cases}$$
                 
\item If $m$ is an even number, then
$$\chi_{A_n}(g)= \begin{cases}
                  m \quad \quad \text{if } g=1\\
                  2 \quad \quad \text { if } g=S_{2i}, ~ 0\leq i\leq m-1\\
                  0 \quad \quad \text {  otherwise}
                 \end{cases}$$
and 
$$\chi_{B_n}(g)= \begin{cases}
                  m \quad \quad \text{if } g=1\\
                  2 \quad \quad \text { if } g=T_{2i+1}, ~ 0\leq i\leq m-1\\
                  0 \quad \quad \text {  otherwise}
                 \end{cases}$$
\end{enumerate}
\end{lemma}
\begin{proof} 
We identify ${\rm Inn}(A_n)$ and ${\rm Inn}(B_n)$ with $D_m:=\{ r,s ~|~ r^m =s^2=1, srs^{-1}=r^{-1} \}$ (see Remark \ref{remark}). Since characters are class functions, it is sufficient to compute them on representatives of conjugacy classes of the group. 

For convenience, we divide rest of the proof into two cases.

\begin{enumerate}[(1).]
  \item {\bfseries Case $m$ = odd}.  In this case, $D_m$ has $(m+3)/2$ conjugacy classes, namely 
  $$\{1\}, \quad \{r^k, r^{-k}\}; 1 \leq k \leq (m-1)/2, \quad \{ r^ls; 0\leq l \leq m-1\}$$
 Using the proof of Lemma \ref{dihedral} and Remark \ref{remark}, 
 $$\chi_{A_n}(1)= \chi_{B_n}(1)=m, \chi_{A_n}(r^k)= \chi_{B_n}(r^k)=0$$ for all $1 \leq k \leq (m-1)/2$, and $\chi_{A_n}(s)= \chi_{B_n}(s)=1.$ \\

\item {\bfseries Case $m$ = even}. Let $m = 2t, t\in \mathbb{N}.$ In this case, $D_m$ has $t+3$ conjugacy classes, namely 
  $$\{1\}, \quad \{r^t\}, \quad \{r^k, r^{-k}\}; 1 \leq k \leq t-1,  \quad \{ r^ls;  l \text{ even}\}, \quad \{ r^ls;  l \text{ odd}\}$$
Again, from the proof of Lemma \ref{dihedral} and Remark \ref{remark}, 
$$\chi_{A_n}(1)= \chi_{B_n}(1)=m, \chi_{A_n}(r^k)= \chi_{B_n}(r^k)=0$$ for all $1 \leq k \leq t.$ We compute $\chi_{A_n}(s)= 2$ and  $\chi_{B_n}(s)=0$, whereas $\chi_{A_n}(rs)= 0$ and  $\chi_{B_n}(rs)=2.$ This completes proof of this lemma.
\end{enumerate}
\end{proof}

Let $\displaystyle \epsilon$ denote a primitive $n^{th}$ root of unity. Let $K= \mathbb{R}$ or $\mathbb{C}$ and $V_j$ be a two dimensional vector space over $K$ for each integer $j$ with $1\leq j \leq m-1.$ Let $D_m:=\langle r,s ~|~ r^m =s^2=1, srs^{-1}=r^{-1} \rangle$ be the dihedral group of order $2m.$ Then the map $\psi_j : D_m \to {\rm GL}(V_j)$ defined by 
$$ \psi_j(r)=\begin{pmatrix}
              \epsilon^j &  0\\
              0 &  \epsilon^{-j}
             \end{pmatrix}       \hspace{2cm}  \psi_j(s)=\begin{pmatrix}
              0 &  1\\
              1 &  0
             \end{pmatrix}  $$
             is a representation of $D_m.$ 
             If $m$ is an odd number, then the set $\{ \psi_j ~|~  1\leq j \leq (m-1)/2 \}$ is a complete set of inequivalent irreducible representations of $D_m$ with degree $2$. If $m$ is an even number, then the set $\{ \psi_j ~|~  1\leq j \leq \frac{m}{2}-1\}$ is complete set of inequivalent irreducible representations of $D_m$ with degree $2$. For more details, we refer the reader to \cite[\S 18.3]{JL}.

\begin{lemma}\label{InnerProduct}
 Let $k= (m-1)/2$ if $m$ is odd number and $k=\frac{m}{2}-1$ if $m$ is an even number. Then for each integer $j$ such that $1\leq j\leq k,$ the representation $\psi_{j}$ appears in the decomposition of $\rho_{A_n}$ as well as $\rho_{B_n}$ with multiplicity $1.$  
\end{lemma}
\begin{proof}
 Let $\chi_j$ denote the character associated to the irreducible representation $\psi_{j}$ of $D_m.$ We recall that $\chi_{A_n}$ is character associated to $\rho_{A_n}$. For $1\leq j\leq k,$ using lemma \ref{Character} and \cite[\S 18.3]{JL}, we compute
  $$ \langle \chi_{A_n}, \chi_j \rangle = \frac{1}{2m}\sum_{g\in {\rm Inn}(A_n)} \chi_{A_n}(g) \chi_j(g)
  = 1.$$
The inner product of $\chi_{B_n}$ and $\chi_j$ is also $1,$ where $\chi_{B_n}$ is the character associated to $\rho_{B_n}$. Now the lemma follows from \cite[Th. 14.17]{JL}.
\end{proof}

\begin{theorem}
 Let $K= \mathbb{C}$ or $\mathbb{R}.$ Let $\mathcal{R}_n$ be the dihedral quandle of order $n.$
 \begin{enumerate}
 \item If $m$ is odd, then $$K[\mathcal{R}_n]= V_{triv, even} \oplus V_{triv, odd} \oplus \bigoplus_{j=1}^{(m-1)/2} V_{j, even} \oplus \bigoplus_{j=1}^{(m-1)/2} V_{j, odd}.$$
 \item If $m=2t$, then $$K[\mathcal{R}_n]= V_{triv, even} \oplus V_{triv, odd} \oplus V_{ref, even} \oplus V_{ref, odd} \oplus \bigoplus_{j=1}^{t-1} V_{j, even} \oplus \bigoplus_{j=1}^{t-1} V_{j, odd}.$$ 
 \end{enumerate}

\end{theorem}

\begin{proof}
 It follows from the Lemma \ref{InnerProduct}, that if $\psi_j : D_m \to {\rm GL}(V_j)$ is an irreducible representation of degree $2,$ then $V_j$ is an indecomposable $K[A_n]$-module with multiplicity $1.$ We denote this module by $V_{j, even}.$ Let us consider the following two cases.

 \begin{enumerate}
  \item {\bfseries $m$ = odd}. In this case, $D_m$ has $(m-1)/2$ inequivalent irreducible representations $\psi_j$ of degree $2.$ Therefore $K[A_n]$ has $(m-1)/2$ distinct indecomposable $K[A_n]$-modules $V_{j, even}$ for $1\leq j\leq (m-1)/2.$ Let $V_{triv, even}$ be one dimensional $K[A_n]$-module generated by  $v_{triv, even}.$ By dimension count, we get 
$$K[A_n]= V_{triv, even} \oplus \bigoplus_{j=1}^{(m-1)/2} V_{j, even}.$$
 
\item {\bfseries $m$= even}. Let $m = 2t$. In this case, $D_m$ has $t-1$ inequivalent irreducible representations $\psi_j$ of degree $2.$ Therefore $K[A_n]$ has $t-1$ distinct indecomposable $K[A_n]$-modules $V_{j, even}$ for $1\leq j\leq t-1$. Let $V_{triv, even}$ be one dimensional $K[A_n]$-module generated by  $v_{triv, even},$ and $V_{ref, even}$ be one dimensional $K[A_n]$-module generated by  $v_{ref, even}$ given in Lemma \ref{ref}. Again, by dimension count,
$$K[A_n]= V_{triv, even} \oplus  V_{ref, even} \oplus\bigoplus_{j=1}^{t-1} V_{j, even}.$$
 \end{enumerate}
One can compute the decomposition of $K[B_n]$ into indecomposable $K[B_n]$-modules along a similar line. Now, the theorem follows from the fact that $K[\mathcal{R}_n]=K[A_n]\oplus K[B_n].$ 
\end{proof}

To conclude the article, we provide a complete decomposition of the quandle ring $K[\mathcal{R}_8].$

\begin{example}
\normalfont
 Let $K=\mathbb{C}$ or $\mathbb{R}.$ We consider the dihedral quandle $\mathcal{R}_8$ of order $8.$  The decomposition of quandle ring $K[\mathcal{R}_8]$ is given below:
 $$K[\mathcal{R}_8]= V_{triv, even} \oplus V_{triv, odd} \oplus V_{ref, even} \oplus V_{ref, odd} \oplus V_{1, even} \oplus V_{1, odd}.$$ 
 
 Let $\{v_0, v_2, v_4, v_6 \}$ and $\{v_1, v_3, v_5, v_7 \}$ be bases for $K[A_8]$ and $K[B_8],$ respectively, obtained by identifying $\{v_0, v_2, v_4, v_6 \}$ and $\{v_1, v_3, v_5, v_7 \}$  with $A_8$ and $B_8,$ respectively. With these notations, the components of decomposition of $K[\mathcal{R}_8]$ are
 \begin{align*}
       V_{triv, even} &= \langle v_0+v_2+v_4+v_6 \rangle \\
       V_{triv, odd} &= \langle v_1+v_3+v_5+v_7 \rangle \\
       V_{ref, even} &= \langle v_0-v_2+v_4-v_6 \rangle \\
       V_{ref, odd} &= \langle v_1-v_3+v_5-v_7 \rangle \\
       V_{1, even} &= \langle v_0-v_4, v_2-v_6 \rangle \\
       V_{1, odd} &= \langle v_1-v_5, v_3-v_7 \rangle.
 \end{align*}

\end{example}

\bibliographystyle{amsalpha}
\bibliography{QuandleRing}

\end{document}